\documentclass[onefignum,onetabnum]{siamart190516}

\usepackage[caption=false]{subfig}
\usepackage{fancyvrb}
\newcommand*{\fvtextcolor}[2]{\textcolor{#1}{#2}}
\definecolor{mygreen}{RGB}{63,126,28}
\definecolor{myyellow}{RGB}{102,102,17}
\usepackage{lipsum}
\usepackage{amsfonts}
\usepackage[normalem]{ulem}
\usepackage{graphicx}
\usepackage{epstopdf}
\usepackage{pifont}
\usepackage{siunitx}
\usepackage{pgfplotstable}
    \usetikzlibrary{
        pgfplots.colorbrewer,
    }
    \pgfplotsset{
        cycle list/.define={my marks}{
            every mark/.append style={solid,fill=\pgfkeysvalueof{/pgfplots/mark list fill}},mark=*\\
            every mark/.append style={solid,fill=\pgfkeysvalueof{/pgfplots/mark list fill}},mark=square*\\
            every mark/.append style={solid,fill=\pgfkeysvalueof{/pgfplots/mark list fill}},mark=triangle*\\
            every mark/.append style={solid,fill=\pgfkeysvalueof{/pgfplots/mark list fill}},mark=diamond*\\
        },
    }
\usepackage{algpseudocode}
\usepackage{multirow}
\usepackage{adjustbox}
\usepackage{booktabs,colortbl}
\usepackage{cprotect}
\usepackage{stmaryrd}

\ifpdf
  \DeclareGraphicsExtensions{.eps,.pdf,.png,.jpg}
\else
  \DeclareGraphicsExtensions{.eps}
\fi

\newsiamremark{remark}{Remark}
\newsiamremark{hypothesis}{Hypothesis}
\crefname{hypothesis}{Hypothesis}{Hypotheses}
\newsiamthm{claim}{Claim}

\headers{Preconditioner for Sparse SPD Matrices}{H. Al Daas and P. Jolivet}

\title{A Robust Algebraic Multilevel Domain Decomposition Preconditioner For Sparse Symmetric Positive Definite Matrices\thanks{Submitted to the editors \today.}}
\author{Hussam Al Daas\thanks{STFC Rutherford Appleton Laboratory, Harwell Campus, Didcot, Oxfordshire, OX11 0QX, UK 
  (\email{hussam.al-daas@stfc.ac.uk}).}
  \and Pierre Jolivet\thanks{CNRS, ENSEEIHT, 2 rue Charles Camichel, 31071 Toulouse Cedex 7, France (\email{pierre.jolivet@enseeiht.fr}).}
}

\usepackage{amsopn}

\newcommand{\R}{\mathbb{R}}

\newcommand{\epart}[1]{\widetilde{\Omega}_{#1}}
\newcommand{\en}{\widetilde{n}}
\newcommand{\eres}[1]{\widetilde{R}_{#1}}
\newcommand{\erest}[1]{\widetilde{R}^\top_{#1}}
\newcommand{\part}[1]{\Omega_{#1}}
\newcommand{\res}[1]{R_{#1}}
\newcommand{\rest}[1]{R^\top_{#1}}
\newcommand{\Z}[1]{Z_{#1}}
\newcommand{\nomega}[1]{n_{#1}}

\newcommand{\schwarz}[1]{M^{-1}_{\text{\tiny{#1}}}}

\pgfplotsset{compat=newest}
\pgfplotsset{colormap={paraview}{rgb(0cm)=(0.278431,0.278431,0.858824) rgb(0.1428571429cm)=(0,0,0.360784) rgb(0.2857142857cm)=(0,1,1) rgb(0.4285714286cm)=(0,0.501961,0) rgb(0.5714285714cm)=(1,1,0) rgb(0.7142857143cm)=(1,0.380392,0) rgb(0.8571428571cm)=(0.419608,0,0) rgb(1cm)=(0.878431,0.301961,0.301961)}}
\pgfplotsset{colormap={parabin}{rgb(0cm)=(0.278431,0.278431,0.858824) rgb(0.5cm)=(0.878431,0.301961,0.301961) rgb(1cm)=(0.878431,0.301961,0.301961)}}
\usepackage{epstopdf}
\DeclareGraphicsExtensions{%
    .png,.PNG,%
    .pdf,.PDF,%
    .jpg,.mps,.jpeg,.jbig2,.jb2,.JPG,.JPEG,.JBIG2,.JB2}
\ifpdf
\hypersetup{
  pdftitle={A Robust Algebraic Multilevel Domain Decomposition Preconditioner For Sparse Symmetric Positive Definite Matrices},
  pdfauthor={H. Al Daas and P. Jolivet}
}
\fi

\begin{document}

\maketitle

% REQUIRED
\begin{abstract}
 Domain decomposition (DD) methods are widely used as preconditioner techniques. Their effectiveness
 relies on the choice of a locally constructed coarse space. Thus far, this construction was mostly
 achieved using non-assembled matrices from discretized partial differential
 equations (PDEs). Therefore, DD methods were mainly successful when solving
    systems stemming from PDEs. In
 this paper, we present a fully algebraic multilevel DD method where the
 coarse space can be constructed locally and efficiently without any
 information besides the coefficient matrix. The condition number of the
 preconditioned matrix can be bounded by a user-prescribed number. Numerical
 experiments illustrate the effectiveness of the preconditioner on a range
 of problems arising from different applications.
\end{abstract}

% REQUIRED

% REQUIRED
%\begin{AMS}
%  68Q25, 68R10, 68U05
%\end{AMS}

\section{Introduction}
\label{sec:introduction}
We are interested in solving the linear system of equations
\begin{equation*}
  Ax=b,
\end{equation*}
where $A\in\R^{n\times n}$ is a sparse symmetric positive definite (SPD) matrix and $b\in\R^n$ is the right-hand side.
On the one hand, despite their accuracy, direct methods \cite{DufER17} that are based on matrix factorizations become memory
and computationally demanding for large-scale problems. Furthermore, establishing a high level of concurrency in their algorithm
is challenging, which limits the effectiveness of their parallelization with many processing units, e.g., thousands of MPI processes.
On the other hand, iterative methods, such as Krylov subspace methods, are attractive as they require less
memory resources and parallelizing them is easier. However, their convergence depends on the coefficient
matrix $A$, the initial guess $x_0$, and the right-hand side $b$. More precisely, the error at iteration $k$ of
the conjugate gradient method \cite{HesS52} satisfies
\[
 \|x_k - x_\star\|_A \le 2 \|x_0 - x_\star\|_A \left(\frac{\sqrt{\kappa_2(A)} - 1}{\sqrt{\kappa_2(A)} + 1}\right)^k,
\]
where $x_\star$ is the exact solution and $\kappa_2(A)$ is the spectral condition number of $A$.
Therefore, iterative methods are usually combined with preconditioners that modifies the properties
of the linear system such that the convergence rate of the methods is improved.
A variety of preconditioning techniques have been proposed in the literature, see the recent survey \cite{PeaP21} and 
references therein.
We focus in this work on preconditioners for SPD matrices.
In terms of construction type, these preconditioners can be split into two categories.
(1) {\it Algebraic} preconditioners: those do not require information from the problem besides the linear system,
and their construction relies only on $A$ and $b$ \cite{AldRS21,HigM19,LiXS16,Not10,Saa03}. (2) {\it Analytic} preconditioners: in order to construct
them, more information from the origin of the linear system, e.g., matrix assembly procedure, is required \cite{JonVVK09,JonVVS13,SpiDHNPS14}.
Inferring how preconditioners modify the spectrum of iteration matrices provides another way to classify them.
Again, two categories exist. (1) {\it One-level} preconditioners: those mostly rely on incomplete matrix
factorizations, matrix splitting methods, approximate sparse inverse methods, and Schwarz methods \cite{Saa03}.
One-level preconditioners usually bound from above the largest eigenvalue of the preconditioned matrix.
(2) {\it Two-level} and {\it multilevel} preconditioners: those are usually a combination of a one-level method and a coarse
space correction. While the one-level part can bound from above the largest eigenvalue, the coarse space is used
to bound from below the smallest eigenvalue such that the condition number of the preconditioned matrix is bounded
\cite{AldG19,AldGJT19,AldJS21,DolJN15,GanL17,GouS21,HeiHK20,HeiKKRW20,KlaKR16,KlaRR15,LiXS16,doi:10.1137/15M1025785,SpiR13,TanNVE09}.

When it comes to overlapping DD, most one-level preconditioners and a few
two-level/multilevel preconditioners are algebraic, while most two-level
preconditioners are analytic.
On the one hand, analytic two-level/multilevel preconditioners construct the coarse space efficiently without requiring computations
involving the global matrix. 
On the other hand, existing algebraic two-level/multilevel preconditioners still require global computations
involving the matrix $A$ that limit the setup scalability \cite{AldG19,GouS21}.
Furthermore, certain algebraic two-level preconditioners require complicated operations that may not be easy to parallelize.
Therefore, we focus in this paper on two-level/multilevel preconditioners where the coarse space can be constructed locally.
Certain algebraic multigrid (AMG) methods are examples of these preconditioners~\cite{Not10}. Note that several AMG methods
require unassembled matrices or the near-nullspace of the global matrix, which is known in some applications~\cite{ChaFHJMMRV03,TamJM15}.
One could argue that these methods are thus not purely algebraic.
Furthermore, their effectiveness has been proved only for certain classes of matrices.
An algebraic two-level preconditioner for the normal matrix equations was recently proposed in~\cite{AldJS21}.

In \cite{AldG19}, the authors presented an algebraic framework to construct robust coarse spaces and characterized
a class of local symmetric positive semi-definite (SPSD) matrices that allows to construct such coarse spaces efficiently.
Since then, there have been several attempts to construct algebraic two-level preconditioners with a locally computed coarse space
that are theoretically effective on any sparse SPD matrix, see, e.g., \cite{GouS21} and references therein.
Starting off with the subdomain matrices of $A$, the authors in \cite{GouS21} define an auxiliary matrix $A_+$ such that
$A- A_+$ is low-rank and a local SPSD splitting for $A_+$ is easily obtained.
A robust algebraic two-level preconditioner for $A$ is then derived by a low-rank update of the robust algebraic two-level 
preconditioner of $A_+$.
Despite the fact that the preconditioner proposed in \cite{GouS21} is fully algebraic, using it in practice may not be
very attractive since the low-rank update requires the solution of linear systems with $A_+$ involving a large number of
right-hand sides that is nearly equal to the size of the coarse space of $A_+$, which is prohibitive for large number of subdomains.
Therefore, we believe that the question of finding efficient locally constructed coarse spaces is still open.

When information such as the near-nullspace or the subdomain non-assembled matrices are available, analytic AMG or DD preconditioners are optimal.
The preconditioner presented in this paper should be used when a robust black-box solver is needed.

The manuscript is organized as follows.
We introduce the notation and review the algebraic DD framework in \Cref{sec:DD}.
\Cref{sec:splitting} presents our main contribution in finding local SPSD splitting matrices associated with each subdomain
fully algebraically in an inexpensive way and starting from local data. These matrices will be used to construct a robust
two-level Schwarz preconditioner. Then, we briefly discuss the straightforward extension of our approach to a 
multilevel preconditioner. Afterwards, we present in \Cref{sec:numerical_experiments} numerical experiments on problems 
arising from different engineering applications. Concluding remarks and future lines of research are given in \Cref{sec:conclusion}.

\paragraph{Notation}
We end our introduction by defining notations that will be used in this paper.
Let $1 \le n \le m$ and let 
$B \in \mathbb{R}^{m \times n}$. 
Let $S_1 \subset \llbracket 1, m \rrbracket$ and $S_2 \subset \llbracket 1, n\rrbracket$ be two sets of integers.
$B(S_1, :)$ is the submatrix of $B$ formed by the rows whose indices belong to
$S_1$ and $B(:, S_2)$ is the submatrix of $B$ formed by the columns whose indices belong to $S_2$.
The matrix $B(S_1,S_2)$ is formed by taking the rows whose indices belong
to~$S_1$ and only retaining the columns whose indices belong to $S_2$.
The concatenation
of any two sets of integers $S_1$ and $S_2$ is represented by $[S_1, S_2]$. Note that the order of the concatenation is
important. The set of the first $p$ positive integers is denoted by $\llbracket 1,p\rrbracket$. 
The identity matrix of size $n$ is denoted by $I_n$.
% We denote by $ker(B)$ and $range(B)$  the nullspace and the range of $B$, respectively.

\section{Domain decomposition}
\label{sec:DD}
Throughout this section, we assume that $C$ is a general $n \times n$ sparse SPD matrix.
Let the nodes $V$ in the corresponding adjacency graph ${\cal G}(C)$ be numbered from $1$ to $n$. 
A graph partitioning algorithm can be used to split $V$ into $N \ll n$ disjoint subsets $\part{Ii}$ ($1\le i \le N$)
of size $\nomega{Ii}$. These sets are called  nonoverlapping subdomains.
\subsection{Abstract setting for two-level overlapping Schwarz methods}
Defining first a one-level Schwarz preconditioner requires overlapping subdomains.
Let $\part{\Gamma i}$ be the subset of size  $\nomega{\Gamma i}$ of  nodes that are distance one in ${\cal G}(C)$ from 
the nodes in $\part{Ii}$  ($1\le i \le N$).
The overlapping subdomain $\part{i}$ is defined to be $\part{i}=[\part{Ii}, \part{\Gamma i}]$,
with size $\nomega{i} = \nomega{\Gamma i} + \nomega{Ii}$.
The complement of $\part{i}$ in $\llbracket1,n\rrbracket$ is denoted by $\part{\text{c}_{\Gamma} i}$.

Associated with $\part{i}$ is a  restriction (or projection) matrix 
$\res{i}\in \mathbb{R}^{n_i \times n}$ given by
$\res{i} = I_n(\part{i},:)$.
$\res{i}$ maps from the global domain to subdomain $\part{i}$. Its transpose
$\rest{i}$ is a prolongation matrix that maps from subdomain $\part{i}$ to the global domain.

The theory in this paper requires a decomposition of the graph of $C^2$.
Hence, in addition to the previous subsets, we define the following ones.
We denote $\part{\Delta i}$ the subset of size  $\nomega{\Delta i}$ containing nodes that are not in $\part{Ii}$ and distance one in ${\cal G}(C)$ from
the nodes in $\part{\Gamma i}$  ($1\le i \le N$).
The extended overlapping subdomain $\epart{i}$ is defined to be $\epart{i} = [\part{Ii}, \part{\Gamma i}, \part{\Delta i}]$ and it is of size $\en_i$.
We denote the complement of $\epart{i}$ in $\llbracket1,n\rrbracket$ by $\part{\text{c}_{\Delta} i}$.
Associated with $\epart{i}$ is a  restriction matrix 
$\eres{i}\in \mathbb{R}^{\en_i \times n}$ given by
$\eres{i} = I_n(\part{i},:)$.
$\eres{i}$ maps from the global domain to the extended overlapping subdomain $\epart{i}$. Its transpose
$\erest{i}$ is a prolongation matrix that maps from the extended overlapping subdomain $\epart{i}$ to the global domain.

The {\em one-level additive Schwarz preconditioner} \cite{DolJN15} is defined to be
\begin{equation*}
  \schwarz{ASM} = \sum_{i=1}^N \rest{i} C_{ii}^{-1} \res{i},  \hspace{0.5cm} C_{ii} = \res{i} C \rest{i}.
\end{equation*}
Applying this preconditioner to a vector involves solving concurrent local
problems in the overlapping subdomains.
Increasing $N$ reduces the sizes $n_i$ of the overlapping subdomains, 
leading to smaller local problems and faster computations.
However, in practice, the preconditioned system using $\schwarz{ASM}$ may not be well-conditioned, inhibiting convergence
of the iterative solver. In fact, the local nature of this preconditioner can lead to a deterioration 
in its effectiveness as the number of subdomains increases because of the lack of global 
information from the matrix~$C$~\cite{DolJN15,GanL17}.
To maintain robustness with respect to $N$, 
a coarse space is added 
to the preconditioner (also known as second-level correction) that includes global information.

Let $0 < n_C \ll n$. If $\res{0} \in \mathbb{R}^{n_C\times n}$ is of full row rank, 
the {\em two-level additive Schwarz preconditioner} \cite{DolJN15} is defined to be 
\begin{equation}
  \label{eq:two_level_schwarz}
  \schwarz{additive} = \sum_{i=0}^N \rest{i} C_{ii}^{-1} \res{i} = \rest{0} C_{00}^{-1} \res{0} + \schwarz{ASM}, \hspace{0.5cm} C_{00} = \res{0} C \rest{0}.
\end{equation}
Observe that, since $C$ and $\res{0}$ are of full rank, $C_{00}$ is also of full rank.
For any full rank $R_0$, it is possible
to cheaply obtain upper bounds on the largest eigenvalue 
of the preconditioned matrix, independently of $n$ and $N$ \cite{AldG19}.
However, bounding the smallest eigenvalue is highly dependent on  $R_0$.
Therefore, the choice of $R_0$ is key to 
obtaining a well-conditioned system and building efficient two-level Schwarz preconditioners.
Two-level Schwarz preconditioners have been used to solve a large class of systems arising from
a range of engineering applications (see, for example,
\cite{HeiHK20,JolRZ21,KonC17,MarCJNT20,SmiBG96,VanSG09} and references therein).

Following \cite{AldG19}, we denote by $D_i \in \mathbb{R}^{n_i\times n_i}$ ($1\le i \le N$) 
any non-negative diagonal matrices such that
\begin{equation*}
  \sum_{i=1}^N \rest{i} D_{i} \res{i} = I_n.
\end{equation*}
We refer to $\left(D_i\right)_{1 \le i \le N}$ as an \emph{algebraic partition of unity}.
In \cite{AldG19}, Al Daas and Grigori show how to select local subspaces $\Z{i} \in \R^{n_i \times p_i}$ 
with $p_i \ll n_i$ ($1 \le i \le N$)   
such that, if $\rest{0}$ is defined to be $\rest{0} = [\rest{1}D_1\Z{1},  \ldots,  \rest{N}D_N\Z{N}]$,
the spectral condition number of the preconditioned matrix $\schwarz{additive} C$ is bounded from above
independently of $N$ and $n$.

\subsection{Algebraic local SPSD splitting of an SPD matrix}
We now recall the definition of an algebraic local SPSD splitting of an SPD matrix
given in \cite{AldG19} and generalized in \cite{AldGJT19}.

An {\em algebraic  local SPSD splitting} of the SPD matrix $C$ with respect to the $i$-th  subdomain is defined to be any 
SPSD matrix $\widetilde{C}_i \in \R^{n \times n}$ that satisfies the following
\begin{align*}
  &0 \le u^\top \widetilde{C}_i u \le u^\top C u, \text{\quad for all } u \in \R^n,\\
  &\res{\text{c}_\Gamma i} \widetilde{C}_i = 0.
\end{align*}
We denote the nonzero submatrix of $\widetilde{C}_i$ by $\widetilde{C}_{ii}$ so that 
$$\widetilde{C}_i = \rest{i} \widetilde{C}_{ii} \res{i}.$$
Associated with the local SPSD splitting matrices, we define a multiplicity constant~$k_m$ that satisfies the inequality
\begin{equation}
  \label{eq:sum C-tildeC_ge_0}
  0 \le \sum_{i=1}^N u^\top \widetilde{C}_i u \le k_m u^\top C u, \text{\quad for all } u \in \R^n.
\end{equation}
Note that, for any set of SPSD splitting matrices, $k_m \le N$.

The main motivation for defining splitting matrices is to 
find local seminorms that are bounded from above by the $C$-norm.
These seminorms will be used to determine a subspace that contains 
the eigenvectors of $C$ associated with its smallest eigenvalues.

\section{Local SPSD splitting matrices}
\label{sec:splitting}
In this section we show how to construct local SPSD splitting matrices of a sparse SPD matrix efficiently
using only local subdomain information.
\subsection{From normal equations matrices to general SPD matrices}
In \cite{AldJS21}, the authors presented how to compute local SPSD splitting matrices for the normal equations matrix $C = B^\top B$
where $B\in\R^{m\times n}$.
Considering the case $B=A$, we have $C = A^2$. Thus, provided the theory developed in \cite{AldJS21}, we can compute
local SPSD splitting matrices of $A^2$ efficiently.
Using the permutation matrix $P_i = I(\part{Ii},\part{\Gamma i}, \part{\Delta i}, \part{\text{c}_\Delta i}] ,:)$, we can write
\begin{equation*}
  P_i A P_i^\top = \begin{pmatrix}
    A_{Ii}          & A_{I \Gamma i}      &                     &\\
    A_{\Gamma I i}  & A_{\Gamma i}        & A_{\Gamma \Delta i} & \\
                    & A_{\Delta \Gamma i} & A_{\Delta i}        & A_{\Delta \text{c}_{\Delta} i} \\
                    &                     & A_{\text{c}_{\Delta} \Delta i}      & A_{\text{c}_{\Delta} i}
  \end{pmatrix},
\end{equation*}
and 
\begin{equation*}
  \widetilde{C}_{i} = \erest{i} X_i^\top X_i \eres{i}
\end{equation*}
is a SPSD splitting of $A^2$, where $X_i$ is given as
\begin{equation}
    \label{eq:Xi}
  X_i = \res{i} A \erest{i} = \begin{pmatrix}
    A_{Ii}          & A_{I \Gamma i}      &                     \\
    A_{\Gamma I i}  & A_{\Gamma i}        & A_{\Gamma \Delta i}  \\
  \end{pmatrix}.
\end{equation}
\begin{remark}\label{remark:sub}
    All terms from~\cref{eq:Xi} stem from the original coefficient matrix $A$,
    in the sense that there is no connection with the underlying discretization
    scheme or matrix assembly procedure. In a parallel computing context, e.g.,
    if $A$ is distributed following a contiguous one-dimensional row
    partitioning among MPI processes, all terms may be retrieved using
    peer-to-peer communication between neighboring processes.
\end{remark}
\Cref{lemma:splitting_A2} demonstrates how to obtain a local SPSD splitting of $A$ with respect to the extended overlapping subdomains given a SPSD splitting of $A^2$.
\begin{lemma}
  \label{lemma:splitting_A2}
  Let $\widetilde{C}_i$ be a local SPSD splitting of $C = A^2$, and let $\underline{\widetilde{A}}_i$ be the square root SPSD matrix of $\widetilde{C}_i$ such that $ \underline{\widetilde{A}}_i^2 =  \widetilde{C}_i$. Then, $ \underline{\widetilde{A}}_i$ is a local SPSD splitting of $A$ with respect to the extended overlapping subdomain $\epart{i}$. 
\end{lemma}
\begin{proof}
  First, observe that for any vector $u\in\R^n$,
  \[
    u^\top (A^2 - \underline{\widetilde{A}}_i^2) u = u^\top (A+\underline{\widetilde{A}}_i) (A-\underline{\widetilde{A}}_i) u.
  \]
  Since $A+ \underline{\widetilde{A}}_i$ is SPD, we can write $A+\underline{\widetilde{A}}_i = W_i^\top W_i$, and we have
  \begin{align*}
    u^\top W_i (A - \underline{\widetilde{A}}_i) W_i^{-1} u &= u^\top W_i^{-\top} W_i^\top W_i (A-\underline{\widetilde{A}}_i) W_i^{-1} u\\
                                                &= v^\top W_i^\top W_i (A-\underline{\widetilde{A}}_i) v\\
                                                &= v^\top (A+\underline{\widetilde{A}}_i) (A-\underline{\widetilde{A}}_i) v\\
                                                &= v^\top (A^2 - \underline{\widetilde{A}}_i^2) v\\
                                                &\ge 0,
  \end{align*}
  where $v = W_i^{-1} u$.
  Since $W_i (A - \underline{\widetilde{A}}_i) W_i^{-1}$ and $A - \underline{\widetilde{A}}_i$ have the same eigenvalues, we conclude that
  $A - \underline{\widetilde{A}}_i$ is SPSD.
  The locality of $ \underline{\widetilde{A}}_i$ stems from the locality of $ \widetilde{C}_i$.
\end{proof}

We note that the SPSD splitting $\underline{\widetilde{A}}_i$ obtained from the SPSD splitting of $A^2$ is local with respect
to the extended overlapping subdomain $\epart{i}$.
A Schur complement technique can be applied to obtain the locality to the subdomain $\part{i}$.
\Cref{lemma:splitting_A} presents how to obtain a local SPSD splitting matrix of $A$ with respect to the subdomain $\part{i}$ from the local
SPSD splitting of $A$ with respect to the extended overlapping subdomain $\epart{i}$.
\begin{lemma}
  \label{lemma:splitting_A}
  Let $\underline{\widetilde{A}}_i = \erest{i} \underline{\widetilde{A}}_{ii} \eres{i}$ be a local SPSD splitting of $A$ with respect to the extended overlapping subdomain $\epart{i}$.
  Let $\underline{\widetilde{A}}_{ii}$ be written as a $(2,2)$ block matrix such that the $(1,1)$ block corresponds to the overlapping subdomain $\part{i}$ and the $(2,2)$ block corresponds to $\part{\Delta i}$, i.e.,
  \[
    \underline{\widetilde{A}}_{ii} = \begin{pmatrix} X_{i,11} & X_{i,12} \\ X_{i,21} & X_{i,22} \end{pmatrix},
  \]
  and let 
  \begin{equation}
   \label{eq:tildeA}
   \widetilde{A}_{ii} = X_{i,11} - X_{i,12} X_{i,22}^{-1} X_{i,21},
  \end{equation}
  where we assume that $X_{i,22}$ is SPD. Then,
  $\widetilde{A}_{i} = \rest{i} \widetilde{A}_{ii} \res{i}$ is a SPSD splitting of $A$ with respect to the subdomain $\part{i}$.
\end{lemma}
\begin{proof}
  We have
\begin{align*}
  \underline{\widetilde{A}}_{ii} 
  &= \begin{pmatrix}
       X_{i,11} & X_{i,12} \\
       X_{i,21} & X_{i,22}
     \end{pmatrix}\\
  &= \begin{pmatrix}
    X_{i,11} - X_{i,12} X_{i,22}^{-1} X_{i,21} &  \\
    &
  \end{pmatrix} + 
  \begin{pmatrix}
    X_{i,12} X_{i,22}^{-1} X_{i,21} & X_{i,12}  \\
    X_{i,21} & X_{i,22}
  \end{pmatrix}.
\end{align*}
  Since $X_{i,22}$ is SPD and $\underline{\widetilde{A}}_{ii}$ is SPSD, $X_{i,11} - X_{i,12} X_{i,22}^{-1} X_{i,21}$ is SPSD.
Therefore,
  \begin{align*}
    0 \le u^\top \widetilde{A}_i u &=   u^\top \rest{i} \widetilde{A}_{ii} \res{i} u\\
                                   &\le u^\top \erest{i} \underline{\widetilde{A}}_{ii} \eres{i} u\\
                                   &\le u^\top A u.
  \end{align*}

\end{proof}
\begin{remark}
 \label{remark:shift}
 Since the SPSD splitting will be used to construct a preconditioner, the assumption in \cref{lemma:splitting_A} that $X_{i,22}$ is SPD can be obtained by shifting its diagonal elements by a small value such as $\|X_{i,22}\|_2 \varepsilon$, where $\varepsilon$ is the floating-point machine precision.
One can also shift the diagonal values of the matrix $\underline{\widetilde{A}}_{ii}$ by a small value $\|\underline{\widetilde{A}}_{ii}\|_2 \varepsilon$ so that the Schur complement can be well defined.
\end{remark}
In the following section, we explain how to compute the local SPSD splitting matrices efficiently.
\subsection{Practical construction of local SPSD matrices\label{sec:practical}}
The construction of robust two-level overlapping Schwarz preconditioners is based on computing
the coarse space projection operator $\res{0}$. Using the local SPSD splitting matrices of $A$,
$\res{0}$ can be chosen as the matrix that spans the space
\[
  Z = \bigoplus_{i=1}^N \rest{i}D_i Z_i,
\]
where $Z_i$ is defined to be
\begin{equation}
  \label{eq:Zi}
  Z_i = \text{span}\{u \ | \ D_i A_{ii} D_i u = \lambda \widetilde{A}_{ii} u \text{, and } \lambda > 1/\tau\},
\end{equation}
where $\tau > 0$ is a user-specified number.
The condition number of the preconditioned matrix $\schwarz{additive} A$ is bounded from above by $ (k_c + 1) \left(2 + (2k_c + 1)\frac{k_m}{\tau}\right)$, 
where $k_c$ is the number of colors required to color the graph of $A$ such that any two neighboring subdomains have different colors and $k_m$ 
is the multiplicity constant that satisfies~\cref{eq:sum C-tildeC_ge_0}.

Solving the generalized eigenvalue problem in~\cref{eq:Zi} using iterative solvers such as the Krylov--Schur method~\cite{Ste02} requires solving linear systems of the form $\widetilde{A}_{ii} u = v$.
The matrix $ \widetilde{A}_{ii}$ is the Schur complement of the matrix $\underline{\widetilde{A}}_{ii} = \left(X_i^\top X_i\right)^{\frac{1}{2}}$,
where $X_i = \res{i} A \erest{i}$. 
Let $X_i = U_i \Sigma_i V_i^\top$ be the economic singular-value decomposition of $X_i$ and let $V_i^\perp$ be an orthonormal matrix whose
columns form a complementary basis of the columns of $V_i$, i.e., $[V_i, V_i^\perp]$ is an orthogonal matrix. 
Note that $V_i^\perp (V_i^\perp )^\top = I_{\tilde{n}_i} - V_i V_i^\top$.
Using \cref{remark:shift}, $\underline{\widetilde{A}}_{ii}$ can be chosen as 
\begin{align*}
 \underline{\widetilde{A}}_{ii} &= V_i \Sigma_i V_i^\top + \sigma_{1i} \varepsilon I_{\tilde{n}_i}\\
                                 &= V_i \Sigma_i V_i^\top + \sigma_{1i} \varepsilon [V_i, V_i^\perp] [V_i, V_i^\perp]^\top\\
                                 &= V_i (\Sigma_i+\sigma_{1i} \varepsilon I_{n_i}) V_i^\top + \sigma_{1i} \varepsilon V_i^\perp (V_i^\perp )^\top\\
                                 &= V_i (\Sigma_i+\sigma_{1i} \varepsilon I_{n_i}) V_i^\top + \sigma_{1i} \varepsilon (I_{\tilde{n}_i} - V_i V_i^\top),
\end{align*}
where $\sigma_{1i}$ is the largest singular value of $X_i$.
One way to solve the linear system $\widetilde{A}_{ii} u = v$ is thus to solve the augmented linear system
\[
  \underline{\widetilde{A}}_{ii} \begin{pmatrix} u \\ y\end{pmatrix} = \begin{pmatrix} v\\ 0\end{pmatrix}.
\]
Given the singular-value decomposition of $\underline{\widetilde{A}}_{ii}$, the solution $u$ can be obtained efficiently.
Indeed, the inverse of $\underline{\widetilde{A}}_{ii}$ is 
\begin{equation}\label{eq:invAii}
 \underline{\widetilde{A}}_{ii}^{-1} = V_i (\Sigma_i+\sigma_{1i}\varepsilon I_{n_i})^{-1} V_i^\top + \sigma_{1i}^{-1} \varepsilon^{-1} (I_{\tilde{n}_i} - V_i V_i^\top).
\end{equation}
In our current implementation, the singular-value decomposition is computed concurrently using LAPACK~\cite{lapack99}. This implies that the sparse matrix $X_i$, see~\cref{eq:Xi}, is converted to a dense representation. Then, $\underline{\widetilde{A}}_{ii}$ is never assembled, and instead, the action of its inverse is applied in a matrix-free fashion using~\cref{eq:invAii}. Since these operations are local to each subdomain, they remain tractable. However, it could be beneficial to leverage the lower memory-footprint of iterative sparse singular-value solvers, e.g., PRIMME\_SVDS~\cite{doi:10.1137/16M1082214}. To the best of our knowledge, no such solver may be used to retrieve the complete economic singular-value decomposition of a sparse matrix.

Since the construction of the two-level method is fully algebraic, one can successively apply the same approach on the coarse space matrix to obtain a
multilevel preconditioner in which, the condition number of each preconditioned matrix is bounded from above by a prescribed number.
Note that if the matrices $\widetilde{A}_{ii}$ for $i=1, \ldots, N$ are formed explicitly as in~\cref{eq:tildeA}, we can use the strategy that we proposed in
\cite{AldGJT19} to construct a multilevel preconditioner with the same properties.

\section{Numerical experiments}
\label{sec:numerical_experiments}
In this section, we present a variety of numerical experiments that show the effectiveness
and efficiency of the proposed preconditioner.
First, we compare it against state-of-the-art algebraic multigrid
preconditioners including AGMG~\cite{napov2012algebraic,Not10}, BoomerAMG~\cite{FalY02}, and GAMG~\cite{AdaBKP04}.
Then, we include numerical experiments where the proposed preconditioner is used to solve coarse problems
from other multilevel solvers, thus emphasizing the algebraic and robust traits of our method.
Except for AGMG which is used through its MATLAB interface, all these experiments are performed using PETSc~\cite{PETSc}. In particular, the proposed preconditioner is a natural extension of the PCHPDDM infrastructure~\cite{JolRZ21} which we use to solve the concurrent generalized eigenvalue problems from~\cref{eq:Zi} via SLEPc~\cite{HerRV05}, and then to define our multilevel preconditioner by selecting the appropriate local eigenmodes depending on the user-specified value of $\tau$. With respect to~\cref{remark:sub}, we use the PETSc routine \texttt{MatCreateSubMatrices}, see \url{https://petsc.org/release/docs/manualpages/Mat/MatCreateSubMatrices.html}. Instead of using $\schwarz{additive}$ as defined in~\cref{eq:two_level_schwarz}, we will use $\schwarz{deflated}$, defined as
\begin{equation*}
    \schwarz{deflated} = \rest{0} C_{00}^{-1} \res{0} + \schwarz{RAS}(I_n - C\rest{0} C_{00}^{-1} \res{0}),
\end{equation*}
where $\schwarz{RAS}$ is the well-known one-level restricted additive Schwarz method~\cite{CaiS99}. The choice of $\schwarz{deflated}$ over $\schwarz{additive}$ is motived by previous results from the literature~\cite{TanNVE09}, which exhibit better numerical property of the former over the latter.
\Cref{tab:data_set} presents the set of test matrices from the SuiteSparse Matrix Collection \cite{DavH11} that are used first. They represent a subset of the matrices from the collection which satisfy both criteria ``\emph{Special Structure} equal to \emph{Symmetric}'' and ``\emph{Positive Definite} equal to \emph{Yes}''. We highlight the fact that our proposed preconditioner can handle unstructured systems, not necessarily stemming from standard PDE discretization schemes, by displaying some nonzero patterns in~\cref{fig:pattern}.
\pgfplotstableread{table.dat}\loadedtable
\begin{table}
  \caption{Test matrices taken from the SuiteSparse Matrix Collection.}
  \centering
  \label{tab:data_set}
\pgfplotstabletypeset[every head row/.style={before row=\hline,after row=\hline},
                      every last row/.style={after row=\hline},
                      every even row/.style={before row={\rowcolor[gray]{0.9}}},
                      columns={identifier,m,nnz,condest},sort,sort key=m,%,SPQR
                      columns/m/.style={int detect,column name=$n$,dec sep align},
                      columns/c/.style={int detect,column name=$n_C$,dec sep align},
                      columns/condest/.style={precision=1,column name=condest($A$)},
                      columns/nnz/.style={int detect,column name=nnz($A$),dec sep align},
                      columns/identifier/.style={string type},display columns/0/.style={column name=Identifier, column type={l|}}]\loadedtable
\end{table}
\begin{figure}
    \subfloat[s3rmt3m3, $n = \pgfmathprintnumber{5357}$]{
\begin{tikzpicture}
    \begin{axis}[enlargelimits=false, axis on top, axis equal image, width=6.1cm,yticklabels={,,},xticklabels={,,}]
\addplot graphics [xmin=1,xmax=2,ymin=1,ymax=2] {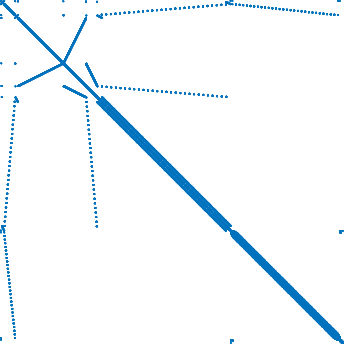};
\end{axis}
\end{tikzpicture}
    }
    \subfloat[ct20stif, $n = \pgfmathprintnumber{52329}$]{
\begin{tikzpicture}
    \begin{axis}[enlargelimits=false, axis on top, axis equal image, width=6.1cm,yticklabels={,,},xticklabels={,,}]
\addplot graphics [xmin=1,xmax=2,ymin=1,ymax=2] {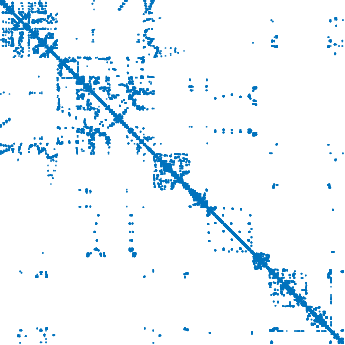};
\end{axis}
\end{tikzpicture}
    }
    \subfloat[finan512, $n = \pgfmathprintnumber{74752}$]{
\begin{tikzpicture}
    \begin{axis}[enlargelimits=false, axis on top, axis equal image, width=6.1cm,yticklabels={,,},xticklabels={,,}]
\addplot graphics [xmin=1,xmax=2,ymin=1,ymax=2] {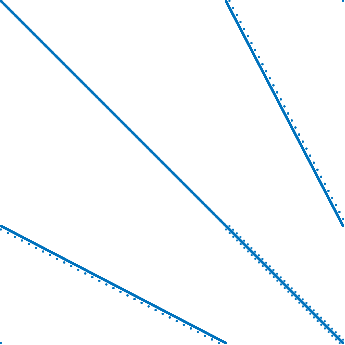};
\end{axis}
\end{tikzpicture}
    } \\
    \subfloat[consph, $n = \pgfmathprintnumber{83334}$]{
\begin{tikzpicture}
    \begin{axis}[enlargelimits=false, axis on top, axis equal image, width=6.1cm,yticklabels={,,},xticklabels={,,}]
\addplot graphics [xmin=1,xmax=2,ymin=1,ymax=2] {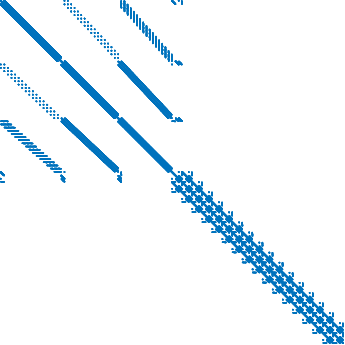};
\end{axis}
\end{tikzpicture}
    }
    \subfloat[G2\_circuit, $n = \pgfmathprintnumber{150102}$]{
\begin{tikzpicture}
    \begin{axis}[enlargelimits=false, axis on top, axis equal image, width=6.1cm,yticklabels={,,},xticklabels={,,}]
\addplot graphics [xmin=1,xmax=2,ymin=1,ymax=2] {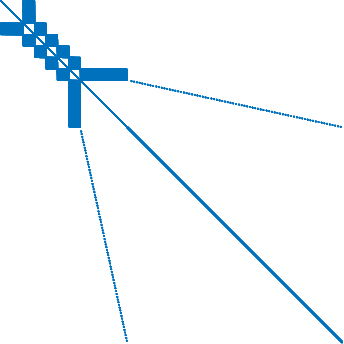};
\end{axis}
\end{tikzpicture}
    }
    \subfloat[offshore, $n = \pgfmathprintnumber{259789}$]{
\begin{tikzpicture}
    \begin{axis}[enlargelimits=false, axis on top, axis equal image, width=6.1cm,yticklabels={,,},xticklabels={,,}]
\addplot graphics [xmin=1,xmax=2,ymin=1,ymax=2] {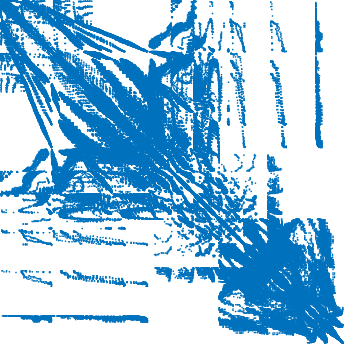};
\end{axis}
\end{tikzpicture}
    }
    \caption{Nonzero sparsity pattern of some of the test matrices from~\cref{tab:data_set}.\label{fig:pattern}}
\end{figure}
\subsection{The algebraic two-level case\label{sec:two}}
In this section, we present a numerical comparison between our proposed preconditioner and three algebraic multigrid solvers: AGMG, BoomerAMG, and GAMG.
Even though matrices from~\cref{tab:data_set} are SPD, all three AMG solvers encounter difficulties in solving many of the associated linear systems with random right-hand sides.
On the contrary, our algebraic two-level preconditioner is more robust and always reach the prescribed tolerance of $10^{-8}$. Note that a simple one-level preconditioner such as $\schwarz{RAS}$ with a minimal overlap of one does not converge for these problems. The outer Krylov method is the right-preconditioned GMRES(30)~\cite{SaaS86}. For preconditioners used within PETSc (all except AGMG), the systems are solved using 256 MPI processes and are first renumbered by ParMETIS~\cite{KarK98}. For our DD method, a single subdomain is mapped to each process, i.e., $N = 256$ in~\cref{eq:two_level_schwarz}. Furthermore, exact subdomain and second-level operator Cholesky factorizations are computed. In the last column of~\cref{tab:comparison}, the size of second-level is reported. One may notice that the grid complexities fluctuate among matrices. Indeed, for small-sized problem s3rmt3m3, the grid complexity is $\frac{\pgfmathprintnumber{5357}+\pgfmathprintnumber{5321}}{\pgfmathprintnumber{5357}}=1.99$, while for problem parabolic\_fem, it is $\frac{\pgfmathprintnumber{525825}+\pgfmathprintnumber{21736}}{\pgfmathprintnumber{525825}}=1.04$.
\begin{table}
  \caption{Preconditioner comparison:
    iteration counts are reported in the columns 2--5 if convergence to the
    prescribed tolerance of $10^{-8}$ is achieved in 100 iterations or less. In
    column 6, sizes of the second-level operator generated by our proposed
    preconditioner are reported.}
  \label{tab:comparison}
  \centering
\pgfplotstabletypeset[every head row/.style={before row=\hline,after row=\hline},
                      every last row/.style={after row=\hline},
                      every even row/.style={before row={\rowcolor[gray]{0.9}}},
                      columns/AGMG/.style={string replace={nan}{}},
                      columns/BoomerAMG/.style={string replace={nan}{}},
                      columns/GAMG/.style={string replace={nan}{}},
                      columns={identifier,AGMG,BoomerAMG,GAMG,HPDDM,c},sort,sort key=m,%,SPQR
                      columns/c/.style={int detect,column name=$n_C$,dec sep align},
                      columns/identifier/.style={string type},display columns/0/.style={column name=Identifier, column type={l|}}]\loadedtable
\end{table}

\subsection{The nested-level case}
Since our proposed preconditioner is fully algebraic, we now use it recursively
to solve the second-level operator from the previous section using yet another two-level
method instead of using an exact Cholesky factorization. This thus yields an algebraic
three-level preconditioner. HPDDM has the capability of automatically
redistributing coarse operators on a subset of MPI processes on which the
initial coefficient matrix $A$ is distributed~\cite{HPDDM}. We still use 256 MPI processes
for the fine-level decomposition, then use four processes for the second-level
decomposition, and the third-level operator is centralized on a single process.
The outer solver is now the flexible GMRES(30)~\cite{doi:10.1137/0914028}. Second-level systems are this
time solved with the right-preconditioned GMRES(30), with a higher tolerance
set to $10^{-4}$, compared to the outer-solver tolerance of $10^{-8}$. We
investigate problems s3rmt3m3 and parabolic\_fem which are the two extremes
from the previous section in terms of grid complexity. Iteration counts are
reported in~\cref{tab:ml}. One may notice that the
number of outer iterations is exactly the same as in the fifth column
of~\cref{tab:comparison}, meaning that the switch to an inexact second-level
solver does not hinder the overall convergence. Also, the number of inner
iterations is small, so our proposed preconditioner applied to the second-level
operator is indeed robust. Eventually, as we decrease the number of subdomains
for the second-level decomposition, the grid coarsening improves as well, especially
for small-sized problem s3rmt3m3.

In another context, we use our proposed preconditioner to solve coarse systems yield by
two other multilevel preconditioners.
The following three-dimensional problems are discretized by FreeFEM~\cite{Hec12} using
\pgfmathprintnumber{4096} MPI processes. First, we use GenEO~\cite{SpiDHNPS14} to
assemble a two-level analytic preconditioner for a scalar diffusion equation using order-two Lagrange
finite elements. The number of unknowns is \pgfmathprintnumber{4173281},
and the second-level operator generated by GenEO is of dimension $n_{C,2} =
\pgfmathprintnumber{60144}$. It is redistributed among 512 processes, and our
preconditioner constructs a third-level operator of dimension $n_{C,3} =
\pgfmathprintnumber{12040}$. Then, we use GAMG to assemble a four-level
quasi-algebraic (the near-nullspace is provided by the discretization kernel) preconditioner for the system of linear elasticity using order-two Lagrange finite elements.
The number of unknowns is \pgfmathprintnumber{30633603}. The coarse
operator from GAMG grid hierarchy is of dimension $n_{C,2} =
\pgfmathprintnumber{14880}$.
It is redistributed among 256 processes using the telescope infrastructure~\cite{10.1145/2929908.2929913}
and our preconditioner constructs a final-level
operator of dimension $n_{C,3} = \pgfmathprintnumber{5120}$. 
Unlike what is traditionally done with smoothed-aggregation AMG~\cite{vanek1992acceleration}, we do
not transfer explicitly the near-nullspace from GAMG coarse level for setting up our

preconditioner. These results are
gathered in~\cref{tab:ml_hybrid}. Again, one may notice that the
fast and accurate convergence of the inner solves (third column) does not
hinder the overall convergence (second column). For both the scalar diffusion
equation $\nabla \cdot \kappa \nabla$ and the system of linear elasticity,
highly heterogeneous material coefficients are used, see \cref{fig:kappa} and
\cref{fig:elasticity}, respectively.
\begin{figure}
\begin{center}
    \subfloat[Scalar diffusion in the unit cube with the coefficient $\kappa$ extruded in one dimension.\label{fig:kappa}]{
\raisebox{-0.2cm}{
    \fbox{\includegraphics[width=0.3\textwidth]{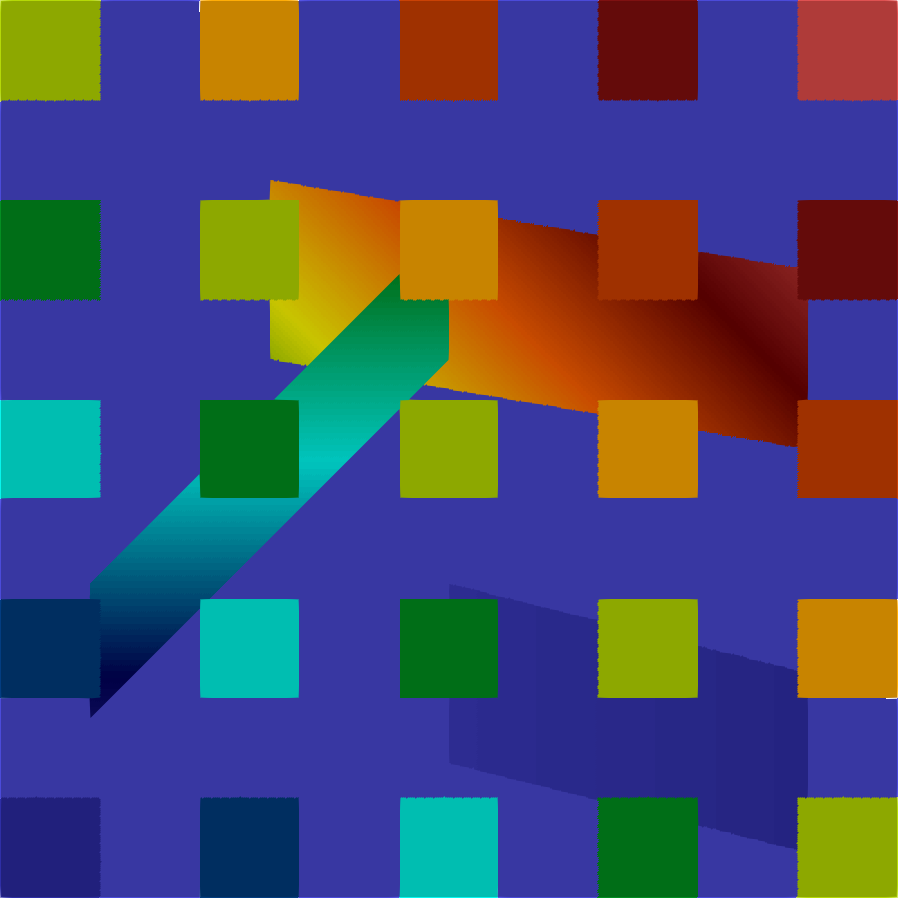}}
        \raisebox{1cm}{\scalebox{0.7}{
\begin{tikzpicture}
	  \begin{axis}[height=4.0cm,
	  enlargelimits=false,
	  colorbar,
	  colormap name=paraview,
	  point meta min=1,
	  point meta max=1.7e+6,
      hide axis,
      colorbar style={
          title={$\kappa$},
          scaled y ticks = false,
          ytick={1,5e+5,1e+6,1.7e+6}
      }
	  ]
	  \end{axis}
	\end{tikzpicture}
}}
    } 
    }\hfill
    \subfloat[Elongated (10$\times$ ratio) three-dimensional beam with Young's modulus ($E$) and Poisson's ratio ($\nu$) extruded in one dimension.\label{fig:elasticity}]{
        \shortstack{\fbox{\includegraphics[width=0.48\textwidth]{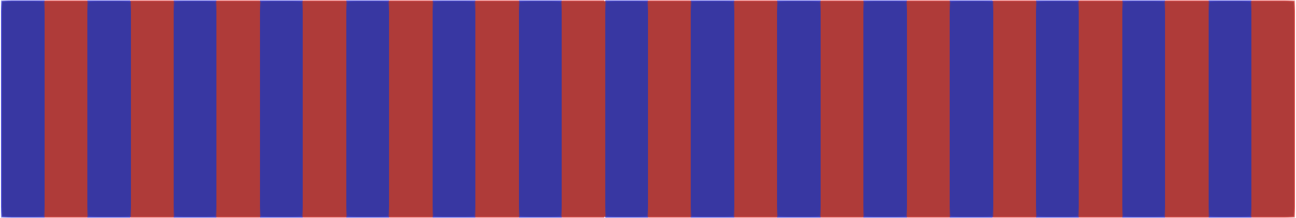}} \\
        \scalebox{0.7}{
\begin{tikzpicture}
	  \begin{axis}[height=4.0cm,
	  enlargelimits=false,
	  colorbar,
	  colormap name=parabin,
	  point meta min=0.01,
	  point meta max=200,
      hide axis,
      colorbar sampled,
      colormap access=piecewise constant,
      colorbar style={
          samples=3,
          title={$E$ (\si{\giga\pascal})},
          scaled y ticks = false,
          ytick={0.01,200}
      }
	  ]
	  \end{axis}
	\end{tikzpicture}
}
        \scalebox{0.7}{
\begin{tikzpicture}
	  \begin{axis}[height=4.0cm,
	  enlargelimits=false,
	  colorbar,
	  colormap name=parabin,
	  point meta min=0.25,
	  point meta max=0.45,
      hide axis,
      colorbar sampled,
      colormap access=piecewise constant,
      colorbar style={
          title={$\nu$},
          samples=3,
          scaled y ticks = false,
          ytick={0.25,0.45}
      }
	  ]
	  \end{axis}
	\end{tikzpicture}
}
    }
    }
    \caption{Variations of the material coefficients for problems from~\cref{tab:ml_hybrid}.\label{fig:coefficients}}
\end{center}
\end{figure}

\pgfplotstableread{table_ml.dat}\loadedtable
\begin{table}
    \caption{Algebraic multilevel preconditioner: {\it Outer iterations} is the FGMRES iteration count, {\it Inner iterations} is the average GMRES iteration count to solve coarse systems, $n$ is the size of the linear system, $n_{C,2}$ (resp. $n_{C,3}$) is the size of the second-level (resp. third-level) operator.}
  \label{tab:ml}
  \centering
\pgfplotstabletypeset[every nth row={2}{before row=\midrule},
                      every head row/.style={before row=\hline,after row=\hline},
                      every last row/.style={after row=\hline},
                      % every even row/.style={before row={\rowcolor[gray]{0.9}}},
                      columns={identifier,outer,inner,m,m2,m3},sort,sort key=m,%,SPQR
                      columns/m/.style={int detect,column name=$n$,dec sep align},
                      columns/outer/.style={int detect,column name=\parbox[c]{2cm}{\small \centering~\vspace*{-2pt}\\ Outer \\ iterations\\[-7pt]~}},
                      columns/inner/.style={int detect,column name=\parbox[c]{2cm}{\small \centering~\vspace*{-2pt}\\ Inner \\ iterations\\[-7pt]~}},
                      columns/m2/.style={int detect,column name=$n_{C,2}$,dec sep align},
                      columns/m3/.style={int detect,column name=$n_{C,3}$,dec sep align},
                      row predicate/.code={%
                          \ifnum#1>4\relax
                          \pgfplotstableuserowfalse
                          \fi},
                      columns/identifier/.style={string type},display columns/0/.style={column name=Identifier, column type={l|}}]\loadedtable
\end{table}

\pgfplotstableread{table_ml_hybrid.dat}\loadedtable
\begin{table}
    \caption{Hybrid multilevel preconditioner: {\it Outer iterations} is the FGMRES iteration count, {\it Inner iterations} is the average GMRES iteration count to solve coarse systems, $n$ is the size of the linear system, $n_{C,2}$ is the size of the coarse-level operator assembled by either GenEO (for problem diffusion) or GAMG (for problem elasticity), $n_{C,3}$ is the size of the second-level operator assembled by our algebraic preconditioner to solve the aforementioned coarse systems.}
  \label{tab:ml_hybrid}
  \centering
\pgfplotstabletypeset[every nth row={2}{before row=\midrule},
                      every head row/.style={before row=\hline,after row=\hline},
                      every last row/.style={after row=\hline},
                      % every even row/.style={before row={\rowcolor[gray]{0.9}}},
                      columns={identifier,outer,inner,m,m2,m3},sort,sort key=m,%,SPQR
                      columns/m/.style={int detect,column name=$n$,dec sep align},
                      columns/outer/.style={int detect,column name=\parbox[c]{2cm}{\small \centering~\vspace*{-2pt}\\ Outer \\ iterations\\[-7pt]~}},
                      columns/inner/.style={int detect,column name=\parbox[c]{2cm}{\small \centering~\vspace*{-2pt}\\ Inner \\ iterations\\[-7pt]~}},
                      columns/m2/.style={int detect,column name=$n_{C,2}$,dec sep align},
                      columns/m3/.style={int detect,column name=$n_{C,3}$,dec sep align},
                      row predicate/.code={%
                          \ifnum#1>4\relax
                          \pgfplotstableuserowfalse
                          \fi},
                      columns/identifier/.style={string type},display columns/0/.style={column name=Identifier, column type={l|}}]\loadedtable
\end{table}
Furthermore, as in~\cref{sec:two}, note that using a simple one-level preconditioner such as $\schwarz{RAS}$ with a minimal overlap of one for solving coarse systems from~\cref{tab:ml,tab:ml_hybrid} does not yield accurate enough inner solutions, thus preventing the outer solvers from converging. Coupling GAMG with our preconditioner is a good assessment of the composability of PETSc solvers~\cite{6341494}, for the interested reader, we provide next in~\cref{fig:options} the exact options used to setup such a multilevel solver.
\begin{figure}
\begin{minipage}[t]{0.37\textwidth}
\begin{Verbatim}[fontsize=\footnotesize,frame=single,framerule=0.1mm,commandchars=&\[\]]
 -ksp_type          fgmres
 -ksp_rtol          1.0e-8
 
 -pc_type           gamg
 -pc_gamg_threshold 0.01
 -pc_gamg_repartition
 -pc_mg_levels      4
 
 -prefix_push mg_coarse_
  -pc_type telescope
  -prefix_push pc_telescope_
   -reduction_factor 16
  -prefix_pop
 -prefix_pop
 &fvtextcolor[mygreen][# continue on the right column]
\end{Verbatim}
\end{minipage}
\begin{minipage}[t]{0.62\textwidth}
\begin{Verbatim}[fontsize=\footnotesize,frame=single,framerule=0.1mm,commandchars=\\\{\},codes={\catcode`$=3\catcode`^=7}]
 \fvtextcolor{mygreen}{# continued from the left column}
 -prefix_push mg_coarse_telescope_
  -ksp_converged_reason
  -ksp_type      gmres
  -ksp_pc_side   right
  -ksp_norm_type unpreconditioned
  -ksp_rtol      1.0e-4

  -pc_type       hpddm
  -prefix_push pc_hpddm_
   -define_subdomains
   -levels_1_pc_type     asm \fvtextcolor{mygreen}{# $\schwarz{RAS}$}
   -levels_1_sub_pc_type cholesky \fvtextcolor{mygreen}{# subdomain solvers}
   -levels_1_eps_nev     20 \fvtextcolor{mygreen}{# smallest $\lambda$ in \cref{eq:Zi}}
   -levels_1_st_type     mat \fvtextcolor{mygreen}{# $\underline{\widetilde{A}}\textsubscript{$ii$}^{\!\!-1}$ from \cref{eq:invAii}}
   -coarse_pc_type       cholesky \fvtextcolor{mygreen}{# coarse solver}
  -prefix_pop
 -prefix_pop
\end{Verbatim}
\end{minipage}
    \caption{PETSc command-line options for coupling GAMG and the proposed preconditioner.\label{fig:options}}
\end{figure}
\section{Conclusion}
\label{sec:conclusion}
We presented in this paper a fully algebraic and locally constructed multilevel overlapping Schwarz preconditioner that can
bound from above the condition number of the preconditioned matrix given a user-defined number.
The construction of the preconditioner relies on finding local SPSD splitting matrices of the matrix $A$.
Computing these splitting matrices involves the computation of the right singular vectors of the local
block row matrix which might be considered costly on the fine level. However, the locality of computations
and the robustness of the preconditioner provide a very powerful and scalable preconditioner that can be used as a
black-box solver especially when other black-box preconditioners fail to achieve a desired convergence rate.
Our implementation is readily available in the PETSc library.
Again, the proposed preconditioner is not meant to replace analytic multilevel preconditioners such as smoothed-aggregation algebraic multigrid
and GenEO. When these work, they will be more efficient algorithmically. However, employing the proposed preconditioner to solve the corresponding coarse problems proved to be effective
and efficient. As a future work, we would like to investigate less expensive constructions of SPSD matrices for
specific classes of SPD matrices that arise from the discretization of PDEs.

\appendix
\section*{Acknowledgments}
This work was granted access to the
GENCI-sponsored HPC resources of TGCC@CEA under allocation A0090607519. 
The authors would like to thank J.~E.~Roman for interesting discussions concerning the solution of~\cref{eq:Zi}.
%
%\section*{Code reproducibility}
%Interested readers are referred to \url{TBA}
%for setting up the appropriate requirements, compiling, and running our proposed
%preconditioner. Fortran, C, and Python source codes are provided.

\bibliographystyle{siamplain}
\bibliography{main}

\begin{thebibliography}{10}

\bibitem{AdaBKP04}
{\sc M.~F. Adams, H.~H. Bayraktar, T.~M. Keaveny, and P.~Papadopoulos}, {\em
  Ultrascalable implicit finite element analyses in solid mechanics with over a
  half a billion degrees of freedom}, in Proceedings of the 2004 ACM/IEEE
  Conference on Supercomputing, SC04, IEEE Computer Society, 2004,
  pp.~\mbox{34:1--34:15}.

\bibitem{AldG19}
{\sc H.~Al~Daas and L.~Grigori}, {\em A class of efficient locally constructed
  preconditioners based on coarse spaces}, SIAM Journal on Matrix Analysis and
  Applications, 40 (2019), pp.~66--91.

\bibitem{AldGJT19}
{\sc H.~Al~Daas, L.~Grigori, P.~Jolivet, and P.-H. Tournier}, {\em A multilevel
  {Schwarz} preconditioner based on a hierarchy of robust coarse spaces}, SIAM
  Journal on Scientific Computing, 43 (2021), pp.~A1907--A1928.

\bibitem{AldJS21}
{\sc H.~{Al Daas}, P.~Jolivet, and J.~A. Scott}, {\em A robust algebraic domain
  decomposition preconditioner for sparse normal equations}, 2021,
  \url{https://arxiv.org/abs/2107.09006}.

\bibitem{AldRS21}
{\sc H.~{Al Daas}, T.~Rees, and J.~A. Scott}, {\em Two-level
  {N}ystr\"om--{S}chur preconditioner for sparse symmetric positive definite
  matrices}, 2021, \url{https://arxiv.org/abs/2101.12164}.

\bibitem{lapack99}
{\sc E.~Anderson, Z.~Bai, C.~Bischof, S.~Blackford, J.~Demmel, J.~Dongarra,
  J.~Du~Croz, A.~Greenbaum, S.~Hammarling, A.~McKenney, and D.~Sorensen}, {\em
  {LAPACK} users' guide}, Society for Industrial and Applied Mathematics, 1999.

\bibitem{PETSc}
{\sc S.~Balay, S.~Abhyankar, M.~F. Adams, J.~Brown, P.~Brune, K.~Buschelman,
  L.~Dalcin, A.~Dener, V.~Eijkhout, W.~D. Gropp, D.~Karpeyev, D.~Kaushik, M.~G.
  Knepley, D.~A. May, L.~C. McInnes, R.~T. Mills, T.~Munson, K.~Rupp, P.~Sanan,
  B.~F. Smith, S.~Zampini, H.~Zhang, and H.~Zhang}, {\em {PETS}c web page},
  2021, \url{https://petsc.org}.

\bibitem{6341494}
{\sc J.~Brown, M.~G. Knepley, D.~A. May, L.~C. McInnes, and B.~F. Smith}, {\em
  Composable linear solvers for multiphysics}, in 2012 11th International
  Symposium on Parallel and Distributed Computing, 2012, pp.~55--62.

\bibitem{CaiS99}
{\sc X.-C. Cai and M.~Sarkis}, {\em A restricted additive {S}chwarz
  preconditioner for general sparse linear systems}, SIAM Journal on Scientific
  Computing, 21 (1999), pp.~792--797.

\bibitem{ChaFHJMMRV03}
{\sc T.~Chartier, R.~D. Falgout, V.~E. Henson, J.~Jones, T.~Manteuffel,
  S.~McCormick, J.~Ruge, and P.~S. Vassilevski}, {\em Spectral {AMGe}
  ($\rho${AMGe})}, SIAM Journal on Scientific Computing, 25 (2003), pp.~1--26.

\bibitem{DavH11}
{\sc T.~A. Davis and Y.~Hu}, {\em The {U}niversity of {F}lorida sparse matrix
  collection}, ACM Transactions on Mathematical Software, 38 (2011), pp.~1--28.

\bibitem{DolJN15}
{\sc V.~Dolean, P.~Jolivet, and F.~Nataf}, {\em An introduction to domain
  decomposition methods. Algorithms, theory, and parallel implementation},
  Society for Industrial and Applied Mathematics, 2015.

\bibitem{DufER17}
{\sc I.~S. Duff, A.~M. Erisman, and J.~K. Reid}, {\em Direct methods for sparse
  matrices}, Oxford University Press, 2017.

\bibitem{FalY02}
{\sc R.~D. Falgout and U.~M. Yang}, {\em \emph{hypre}: a library of high
  performance preconditioners}, Computational Science---ICCS 2002,  (2002),
  pp.~632--641.

\bibitem{GanL17}
{\sc M.~J. Gander and A.~Loneland}, {\em {SHEM}: an optimal coarse space for
  {RAS} and its multiscale approximation}, in Domain Decomposition Methods in
  Science and Engineering XXIII, C.-O. Lee, X.-C. Cai, D.~E. Keyes, H.~H. Kim,
  A.~Klawonn, E.-J. Park, and O.~B. Widlund, eds., Cham, 2017, Springer
  International Publishing, pp.~313--321.

\bibitem{GouS21}
{\sc L.~Gouarin and N.~Spillane}, {\em {Fully algebraic domain decomposition
  preconditioners with adaptive spectral bounds}}.
\newblock Preprint, June 2021,
  \url{https://hal.archives-ouvertes.fr/hal-03258644}.

\bibitem{Hec12}
{\sc F.~Hecht}, {\em New development in {FreeFem++}}, Journal of Numerical
  Mathematics, 20 (2012), pp.~251--265.

\bibitem{HeiHK20}
{\sc A.~Heinlein, C.~Hochmuth, and A.~Klawonn}, {\em Reduced dimension {GDSW}
  coarse spaces for monolithic {Schwarz} domain decomposition methods for
  incompressible fluid flow problems}, International Journal for Numerical
  Methods in Engineering, 121 (2020), pp.~1101--1119.

\bibitem{HeiKKRW20}
{\sc A.~Heinlein, A.~Klawonn, J.~Knepper, O.~Rheinbach, and O.~B. Widlund},
  {\em Adaptive {GDSW} coarse spaces of reduced dimension for overlapping
  {Schwarz} methods}, technical report, Universit{\"a}t zu K{\"o}ln, September
  2020, \url{https://kups.ub.uni-koeln.de/12113/}.

\bibitem{HerRV05}
{\sc V.~Hernandez, J.~E. Roman, and V.~Vidal}, {\em {SLEP}c: a scalable and
  flexible toolkit for the solution of eigenvalue problems}, ACM Transactions
  on Mathematical Software, 31 (2005), pp.~351--362,
  \url{https://slepc.upv.es}.

\bibitem{HesS52}
{\sc M.~R. Hestenes and E.~Stiefel}, {\em Methods of conjugate gradients for
  solving linear systems.}, Journal of research of the National Bureau of
  Standards., 49 (1952), pp.~409--436.

\bibitem{HigM19}
{\sc N.~J. Higham and T.~Mary}, {\em A new preconditioner that exploits
  low-rank approximations to factorization error}, SIAM Journal on Scientific
  Computing, 41 (2019), pp.~A59--A82.

\bibitem{HPDDM}
{\sc P.~Jolivet, F.~Hecht, F.~Nataf, and C.~Prud'homme}, {\em Scalable domain
  decomposition preconditioners for heterogeneous elliptic problems}, in
  Proceedings of the International Conference on High Performance Computing,
  Networking, Storage and Analysis, SC '13, New York, NY, USA, 2013, ACM,
  pp.~80:1--80:11.

\bibitem{JolRZ21}
{\sc P.~Jolivet, J.~E. Roman, and S.~Zampini}, {\em {KSPHPDDM} and {PCHPDDM}:
  extending {PETSc} with advanced {Krylov} methods and robust multilevel
  overlapping {Schwarz} preconditioners}, Computers \& Mathematics with
  Applications, 84 (2021), pp.~277--295.

\bibitem{JonVVK09}
{\sc T.~B. J{\"o}nsth{\"o}vel, M.~B. van Gijzen, C.~Vuik, C.~Kasbergen, and
  A.~Scarpas}, {\em Preconditioned conjugate gradient method enhanced by
  deflation of rigid body modes applied to composite materials}, Computer
  Modeling in Engineering \& Sciences, 47 (2009), pp.~97--118.

\bibitem{JonVVS13}
{\sc T.~B. J{\"o}nsth{\"o}vel, M.~B. van Gijzen, C.~Vuik, and A.~Scarpas}, {\em
  On the use of rigid body modes in the deflated preconditioned conjugate
  gradient method}, SIAM Journal on Scientific Computing, 35 (2013),
  pp.~B207--B225.

\bibitem{KarK98}
{\sc G.~Karypis and V.~Kumar}, {\em Multilevel $k$-way partitioning scheme for
  irregular graphs}, Journal of Parallel and Distributed computing, 48 (1998),
  pp.~96--129.

\bibitem{KlaKR16}
{\sc A.~Klawonn, M.~K{\"u}hn, and O.~Rheinbach}, {\em Adaptive coarse spaces
  for {FETI-DP} in three dimensions}, SIAM Journal on Scientific Computing, 38
  (2016), pp.~A2880--A2911.

\bibitem{KlaRR15}
{\sc A.~Klawonn, P.~Radtke, and O.~Rheinbach}, {\em {FETI-DP} methods with an
  adaptive coarse space}, SIAM Journal on Numerical Analysis, 53 (2015),
  pp.~297--320.

\bibitem{KonC17}
{\sc F.~Kong and X.-C. Cai}, {\em A scalable nonlinear fluid–structure
  interaction solver based on a {Schwarz} preconditioner with isogeometric
  unstructured coarse spaces in {3D}}, Journal of Computational Physics, 340
  (2017), pp.~498--518.

\bibitem{LiXS16}
{\sc R.~Li, Y.~Xi, and Y.~Saad}, {\em {S}chur complement-based domain
  decomposition preconditioners with low-rank corrections}, Numerical Linear
  Algebra with Applications, 23 (2016), pp.~706--729.

\bibitem{MarCJNT20}
{\sc P.~Marchand, X.~Claeys, P.~Jolivet, F.~Nataf, and P.-H. Tournier}, {\em
  Two-level preconditioning for $h$-version boundary element approximation of
  hypersingular operator with {GenEO}}, Numerische Mathematik, 146 (2020),
  pp.~597--628.

\bibitem{10.1145/2929908.2929913}
{\sc D.~A. May, P.~Sanan, K.~Rupp, M.~G. Knepley, and B.~F. Smith}, {\em
  Extreme-scale multigrid components within {PETSc}}, in Proceedings of the
  Platform for Advanced Scientific Computing Conference, PASC '16, New York,
  NY, USA, 2016, Association for Computing Machinery.

\bibitem{napov2012algebraic}
{\sc A.~Napov and Y.~Notay}, {\em An algebraic multigrid method with guaranteed
  convergence rate}, SIAM Journal on Scientific Computing, 34 (2012),
  pp.~A1079--A1109.

\bibitem{Not10}
{\sc Y.~Notay}, {\em An aggregation-based algebraic multigrid method},
  Electronic Transactions on Numerical Analysis, 37 (2010), pp.~123--146,
  \url{http://agmg.eu}.

\bibitem{PeaP21}
{\sc J.~W. Pearson and J.~Pestana}, {\em Preconditioners for {Krylov} subspace
  methods: an overview}, GAMM-Mitteilungen, 43 (2020), p.~e202000015.

\bibitem{doi:10.1137/0914028}
{\sc Y.~Saad}, {\em A flexible inner-outer preconditioned {GMRES} algorithm},
  SIAM Journal on Scientific Computing, 14 (1993), pp.~461--469.

\bibitem{Saa03}
{\sc Y.~Saad}, {\em Iterative methods for sparse linear systems}, Society for
  Industrial and Applied Mathematics, 2003.

\bibitem{SaaS86}
{\sc Y.~Saad and M.~H. Schultz}, {\em {GMRES}: a generalized minimal residual
  algorithm for solving nonsymmetric linear systems}, SIAM Journal on
  Scientific and Statistical Computing, 7 (1986), pp.~856--869.

\bibitem{SmiBG96}
{\sc B.~F. Smith, P.~E. Bj{\o}rstad, and W.~D. Gropp}, {\em Domain
  decomposition: parallel multilevel methods for elliptic partial differential
  equations}, Cambridge University Press, 1996.

\bibitem{SpiDHNPS14}
{\sc N.~Spillane, V.~Dolean, P.~Hauret, F.~Nataf, C.~Pechstein, and
  R.~Scheichl}, {\em Abstract robust coarse spaces for systems of {PDEs} via
  generalized eigenproblems in the overlaps}, Numerische Mathematik, 126
  (2014), pp.~741--770.

\bibitem{SpiR13}
{\sc N.~Spillane and D.~Rixen}, {\em Automatic spectral coarse spaces for
  robust finite element tearing and interconnecting and balanced domain
  decomposition algorithms}, International Journal for Numerical Methods in
  Engineering, 95 (2013), pp.~953--990.

\bibitem{Ste02}
{\sc G.~W. Stewart}, {\em A {K}rylov--{S}chur algorithm for large
  eigenproblems}, SIAM Journal on Matrix Analysis and Applications, 23 (2002),
  pp.~601--614.

\bibitem{TamJM15}
{\sc R.~Tamstorf, T.~Jones, and S.~F. McCormick}, {\em Smoothed aggregation
  multigrid for cloth simulation}, ACM Transactions on Graphics, 34 (2015).

\bibitem{TanNVE09}
{\sc J.~M. Tang, R.~Nabben, C.~Vuik, and Y.~A. Erlangga}, {\em Comparison of
  two-level preconditioners derived from deflation, domain decomposition and
  multigrid methods}, Journal of Scientific Computing, 39 (2009), pp.~340--370.

\bibitem{VanSG09}
{\sc J.~Van~lent, R.~Scheichl, and I.~G. Graham}, {\em Energy-minimizing coarse
  spaces for two-level {Schwarz} methods for multiscale {PDEs}}, Numerical
  Linear Algebra with Applications, 16 (2009), pp.~775--799.

\bibitem{vanek1992acceleration}
{\sc P.~Van{\v{e}}k}, {\em Acceleration of convergence of a two-level algorithm
  by smoothing transfer operators}, Applications of Mathematics, 37 (1992),
  pp.~265--274.

\bibitem{doi:10.1137/16M1082214}
{\sc L.~Wu, E.~Romero, and A.~Stathopoulos}, {\em {PRIMME\_SVDS}: a
  high-performance preconditioned {SVD} solver for accurate large-scale
  computations}, SIAM Journal on Scientific Computing, 39 (2017),
  pp.~S248--S271.

\bibitem{doi:10.1137/15M1025785}
{\sc S.~Zampini}, {\em {PCBDDC}: a class of robust dual-primal methods in
  {PETSc}}, SIAM Journal on Scientific Computing, 38 (2016), pp.~S282--S306.

\end{thebibliography}
\end{document}